\documentclass[12pt,fleqn]{article}
\usepackage{latexsym}
\usepackage{amsmath}
\usepackage{amssymb}
\usepackage{epsfig}
\usepackage{amsfonts}

\begin{document}
\newtheorem{definition}{Definition}[section]
\newtheorem{theorem}[definition]{Theorem}
\newtheorem{lemma}[definition]{Lemma}
\newtheorem{proposition}[definition]{Proposition}
\newtheorem{examples}[definition]{Examples}
\newtheorem{corollary}[definition]{Corollary}
\def\square{\Box}
\newtheorem{remark}[definition]{Remark}
\newtheorem{remarks}[definition]{Remarks}
\newtheorem{exercise}[definition]{Exercise}
\newtheorem{example}[definition]{Example}
\newtheorem{observation}[definition]{Observation}
\newtheorem{observations}[definition]{Observations}
\newtheorem{algorithm}[definition]{Algorithm}
\newtheorem{criterion}[definition]{Criterion}
\newtheorem{algcrit}[definition]{Algorithm and criterion}

\newenvironment{prf}[1]{\trivlist
\item[\hskip \labelsep{\it
#1.\hspace*{.3em}}]}{~\hspace{\fill}~$\square$\endtrivlist}
\newenvironment{proof}{\begin{prf}{Proof}}{\end{prf}}

\title{A Riemann--Hilbert approach to Painlev\'e IV}
\author{ Marius van der Put and Jaap Top}
\date{April 2012}
\maketitle

\begin{abstract} 
This paper applies methods of Van~der~Put and Van~der~Put-Saito to the fourth Painlev\'e equation.
One obtains a Riemann--Hilbert correspondence between moduli spaces of rank two connections on $\mathbb{P}^1$
and moduli spaces for the monodromy data. The moduli spaces for these connections are identified with Okamoto--Painlev\'e 
varieties and the Painlev\'e property follows. For an explicit computation of the full group of B\"acklund transformations, 
rank three connections on $\mathbb{P}^1$ are introduced, inspired by the symmetric form for ${\rm PIV}$ as was studied by M.~Noumi 
and Y.~Yamada.

\end{abstract}

\section*{Introduction}  
In this paper we apply the methods of \cite{vdP-Sa, vdP1, vdP2} to the fourth Painlev\'e equation.  We refer only to a few items of the extensive literature on Okamoto--Painlev\'e varieties. More details on Stokes matrices and the analytic classification of singularities can be found in \cite{vdP-Si}. 

The Riemann--Hilbert approach to the Painlev\'e equation PIV consists of the construction of a moduli space
$\mathcal{M}$ of connections on the projective line and a moduli space $\mathcal{R}$ for the monodromy data.
The Riemann--Hilbert morphism $RH:\mathcal{M}\rightarrow   \mathcal{R}$ assigns to a connection its
monodromy data. The fibres of $RH$, i.e., the isomonodromic families in $\mathcal{M}$, are parametrized
by $t\in T=\mathbb{C}$.   The explicit form of the fibres produces the solutions of PIV. 

$RH^{ext}:\mathcal{M}^+(\theta_0,\theta_\infty)\rightarrow  
 \mathcal{R}^+(\theta_0,\theta_\infty)\times T$, the extended Riemann--Hilbert morphism, is an analytic isomorphism between rather subtle moduli spaces 
 $\mathcal{M}^+(\theta_0,\theta_\infty)$ and $\mathcal{R}^+(\theta_0,\theta_\infty)\times T$, depending on parameters  $\theta_0,\theta_\infty$ and provided with a level structure (or parabolic structure). The Painlev\'e Property for PIV with parameters $\theta_0,\theta_\infty$ follows from this as well as the identification of $\mathcal{M}^+(\theta_0,\theta_\infty)$ with an Okamoto-Painlev\'e variety. 
Formulas for B\"acklund transformations, rational  and Riccati solutions  for PIV are derived.    

 The construction of $\mathcal{M}$ involves the choice of a set $\bf S$ of differential modules over 
 $\mathbb{C}(z)$. In the first part of this paper the `classical' choice for $\bf S$ is treated.
The second choice for $\bf S$ is inspired by the symmetric form for PIV \cite{No,No-Y}, studied by
M.~Noumi and Y.~Yamada. This leads to
a different construction of $\mathcal{M}, \mathcal{R}$ and  Okamoto--Painlev\'e varieties, treated in  Section 3.

\section{The classical choice for  $\bf S$ and  $\mathcal{R}$}

Let $\bf S$ be the set of the isomorphy classes of the differential modules $(M,\delta_M)$ over $\mathbb{C}(z)$
(with  $\delta _M(fm)=(z\frac{d}{dz}f)m+f\delta _M (m)$) having the properties:\\ $\dim M=2$; $\Lambda ^2M$ is trivial; $0,\infty$ are the singular points and the Katz
invariants are $r(0)=0,\ r(\infty )=2$. The variable $z$ is normalized such that the (generalized)
eigenvalues at $\infty$ are $\pm (z^2+\frac{t}{2}z)$. Finally, we exclude the case that $M$ is a direct sum of two
proper submodules since this situation does not produce solutions for PIV.

\bigskip

The {\it monodromy data} at $\infty$ are given by the matrices
 \[{\alpha \ \ 0\choose 0\ \ \frac{1}{\alpha} },\  {1\ \ 0\choose a_1\ \ 1 },\  {1\ \ a_2\choose 0\ \ 1 },\  
 {1\ \ 0\choose a_3 \ \ 1 },\  {1\ \ a_4\choose 0\ \ 1 } \]
with respect to a basis of the symbolic solution space $V(\infty )$ at $z=\infty$ corresponding to
the direct sum expression $V(\infty )=V_{z^2+\frac{t}{2}z}\oplus V_{-(z^2+\frac{t}{2}z)}=\mathbb{C}e_1\oplus \mathbb{C}e_2$. The first matrix is the formal monodromy and the others are the 
four Stokes matrices. The topological monodromy $top_\infty$ at $z=\infty$ (which equals the topological monodromy at $z=0$) is the product of these
matrices in this order. Further we exclude the case $a_1=a_2=a_3=a_4=0$, since this corresponds to the
direct sum situation. The monodromy data form a variety 
$\mathcal{A}:=\mathbb{C}^*\times (\mathbb{C}^4\setminus \{(0,0,0,0)\})$.

The base change $e_1, e_2\mapsto \lambda e_1,\lambda ^{-1}e_2$ induces an
action of $\mathbb{G}_m$ on $\mathcal{A}$. The {\it monodromy space} $\mathcal{R}$ is the quotient
$\mathcal{A}/\mathbb{G}_m$. This quotient can be obtained by gluing the subspaces 
$\mathcal{R}_j, j=1,\dots ,4$ of $\mathcal{A}$, defined by $a_j=1$.  \\

We observe (see \cite{vdP-Sa}, Theorem 1.7) that the map ${\bf S}\rightarrow \mathcal{R}\times T$, which maps a module in $\bf S$
to its monodromy data and value of $t\in T=\mathbb{C}$, is bijective.\\

The {\it parameter space} is $\mathcal{P}=\mathbb{C}\times \mathbb{C}^*$ and $\mathcal{R}\rightarrow \mathcal{P}$ maps an element of $\mathcal{R}$ to $(trace(top_\infty),\alpha)$. The fibre above $(s,\alpha)$ is denoted
by $\mathcal{R}[s,\alpha]$. This fibre is a smooth, connected surface for $s\neq \pm 2$. The fibre
$\mathcal{R}[2,\alpha]$ has {\it one singular point} and this point corresponds to $top_\infty={1\ 0\choose 0 \ 1 }$.
Similarly, $\mathcal{R}[-2,\alpha]$ has {\it one singular point} corresponding to $top_\infty =-{1\ 0\choose 0 \ 1}$.\\

The singular points  are the reason for introducing a {\it level structure} (or `parabolic  structure' in the terminology of \cite{Ina06,IIS1,IIS2,IISA}).
For the monodromy data this is a line $L\subset V(\infty)$ which is invariant under $top_\infty$. The new monodromy space is denoted by $\mathcal{R}^+$. For a module $M$ in $\bf S$ the level structure is a
1-dimensional submodule $N$ of $\mathbb{C}((z))\otimes M$. The submodule $N$ corresponds to an eigenvector
of the topological monodromy $top_0$ at $z=0$ (which is equal to $top_\infty$). The new set is denoted by ${\bf S}^+$. For the parameter space, the level structure
is the introduction of an eigenvalue $\beta$ of $top_\infty$. The new parameter space
$\mathcal{P}^+=\mathbb{C}^*\times \mathbb{C}^*$ maps to $\mathcal{P}$ by $(\beta ,\alpha)\mapsto
(\beta +\beta ^{-1},\alpha)$.  \\

The fibres of $\mathcal{R}^+\rightarrow \mathcal{P}^+$ are denoted by $\mathcal{R}^+(\beta ,\alpha)$.
The morphism  $\mathcal{R}^+(\beta ,\alpha)\rightarrow \mathcal{R}[\beta +\beta ^{-1},\alpha ]$ is an isomorphism for $\beta \neq \pm 1$. A computation shows that

\begin{lemma} $\mathcal{R}^+(\pm 1 ,\alpha)\rightarrow \mathcal{R}[\pm 2,\alpha ]$ is the  minimal resolution.
\end{lemma}
 The map
${\bf S}^+\rightarrow \mathcal{P}^+$ is defined by $\beta =e^{2\pi i\lambda }$ where 
$\delta_M n=\lambda n$ for a basis vector $n$ of $N\subset \mathbb{C}((z))\otimes M$ and $\alpha$ as before. The fibre is written as ${\bf S}^+(\beta , \alpha )$.

\begin{lemma}  The map  ${\bf S}^+(\beta ,\alpha )\rightarrow \mathcal{R}^+(\beta ,\alpha )\times T$ is bijective.
\end{lemma}

The {\it  reducible locus } of $\mathcal{R}^+$ (i.e., the monodromy data is reducible) is the disjoint union of the closed sets defined (in  the notation of $\mathcal{A}$) by $a_2=a_4=0$ and $a_1=a_3=0$. The space $\mathcal{R}^+(\beta ,\alpha )$ contains no reducible elements for $\beta ^{\pm 1}\neq \alpha$. If  $\beta ^{\pm 1} =\alpha $, then the reducible locus of   $\mathcal{R}^+(\beta ,\alpha )$  consists of two non intersecting projective lines.

\section{The moduli space $\mathcal{M}(\theta _0,\theta _\infty)$}
Choose $\theta_0$ with $\beta =e^{\pi i\theta _0}$ and  $\theta _\infty$ with $\alpha =e^{\pi i \theta_\infty}$.
The aim is to replace the set  ${\bf S}^+(\beta ,\alpha )$ by a moduli space of connections 
$\mathcal{M}(\theta _0,\theta_\infty)$ and to study the extended  Riemann--Hilbert map
$RH^{ext}: \mathcal{M}(\theta _0,\theta_\infty)\rightarrow \mathcal{R}^+(\beta ,\alpha)\times T$.

Let a module $(M,N)\in {\bf S}^{+}(\beta,\alpha)$ be given. We define a connection $(\mathcal{W},\nabla )$ on the projective line with $\nabla :\mathcal{W}\rightarrow \Omega ([0]+3[\infty ])\otimes \mathcal{W}$, with generic fibre $M$,  by prescribing the connection $D:=\nabla _{z\frac{d}{dz}}$ locally at $z=0$ as $z\frac{d}{dz}+{\frac{\theta_0}{2}\ \ *\choose \ \ 0\ \ -\frac{\theta_0}{2}}$ and locally at $z=\infty$ as 
$z\frac{d}{dz}+ {\omega \ \ 0 \choose 0\ -\omega }$ with $\omega =z^2+\frac{t}{2}z+\frac{\theta_\infty}{2}$.
This is equivalently to choosing `invariant lattices' at $z=0$ and $z=\infty$. The invariant lattice at $z=0$ is
$\mathbb{C}[[z]]g_1+\mathbb{C}[[z]]g_2\subset \mathbb{C}((z))\otimes M$ with $N=\mathbb{C}((z))g_1$ and
the matrix of $\delta _M$ with respect to $g_1,g_2$ is ${\frac{\theta_0}{2} \ \ \ * \choose \ 0\ \ -\frac{\theta_0}{2} }$.
The invariant lattice at $z=\infty$ is $\mathbb{C}[[z^{-1}]]h_1+\mathbb{C}[[z^{-1}]]h_2\subset \mathbb{C}((z^{-1}\otimes M$ such that $\delta _Mh_1=\omega h_1,\ \delta _Mh_2=-\omega h_2$.\\

 The second exterior power of $(\mathcal{W},\nabla )$ is $d:O\rightarrow \Omega$. Thus $\mathcal{W}$ has degree 0 and type  $O(k)\oplus O(-k)$ with $k\geq 0$. If $M$ is irreducible, then $k\in \{0,1\}$.
 The reducible modules are studied in Observations 2.3.\\ {\it We consider the case $k\in \{0,1\}$}. The connection $(\mathcal{V},\nabla)$,  defined by replacing the invariant lattice $\mathbb{C}[[z]]g_1+\mathbb{C}[[z]]g_2$ by $\mathbb{C}[[z]]g_1+\mathbb{C}[[z]]zg_2$, has type $O\oplus O(-1)$. Further 
 we identify   $\mathcal{V}$ with $Oe_1\oplus O(-[0])e_2$.

 \subsection{The connections on $\mathcal{V}:=Oe_1+O(-[0])e_2$}
 The connection $D=\nabla _{z\frac{d}{dz}}:\mathcal{V}\rightarrow O(2[\infty ])\otimes \mathcal{V}$, obtained from $(M,N)\in {\bf S}^{+}(\beta ,\alpha)$ and the prescribed invariant lattices, has,
 with respect to the basis $e_1,e_2$, the matrix ${a\ \ b \choose c\ -a }$ with
 $a=a_0+a_1z+a_2z^2$, $b=b_{-1}z^{-1}+\cdots +b_2z^2$, $c=c_1z+c_2z^2$. The local data at $z=\infty$
 yields the equations
 \[a_2^2+b_2c_2=1,\ 2a_1a_2+b_2c_1+b_1c_2=t,\ 2a_0 a_2+a_1^2+b_1c_1+b_0c_2=\theta _\infty +\frac{t^2}{4}.\] 
 For $z=0$ one obtains $a_0(a_0-1)+b_{-1}c_1=\frac{\theta_0}{2}(\frac{\theta_0}{2}-1)$.\\

 {\it As a start, we forget the level structure $N$ of the pair $(M,N)\in  {\bf S}^{+}(\beta ,\alpha)$ and assume that
 $c_1z+c_2z^2\neq 0$. } The above variables $a_*,b_*,c_*,t$ and equations define a space $\mathcal{C}$ of dimension 6. We have to divide by the group $G$ of transformations $e_1\mapsto e_1,\ e_2\mapsto \lambda e_2+(x_0+x_1z^{-1})e_1$ of $\mathcal{V}$.
 The quotient $\mathcal{C}/G$ is by definition {\it the moduli space $\mathcal{M}(\theta_0,\theta_\infty)$}.
 
 \begin{proposition}  $\mathcal{M}(\theta_0,\theta_\infty)$ is a  good geometric quotient of $\mathcal{C}$ in the sense
 that there exists a $G$-equivariant isomorphism 
 $G\times \mathcal{M}(\theta_0,\theta_\infty)\rightarrow \mathcal{C}$.\\
 $\mathcal{M}(\theta_0,\theta_\infty)$ is smooth for $\theta_0 \neq 1$. 
For a connection $D \in \mathcal{M}(1,\theta_\infty)$, which is a singular point,  there is a basis of 
$\widehat{\mathcal{V}}_0$ for which $D$ has the form $z\frac{d}{dz}+{\frac{1}{2}\ \ 0\choose 0\ \ \frac{1}{2} }$.     
 \end{proposition}
 \begin{proof}  The {\it  `first standard form'}  $ST_1$ is the closed subset  of $\mathcal{C}$ defined by:\\ 
 $z\frac{d}{dz}+{a_2z^2\ \ b\choose z+c_2z^2 \ -a_2z^2 }$  with $b_2=-c_2(\theta _\infty +\frac{t^2}{4}-b_0c_2)+t$, $b_1=\theta _\infty +\frac{t^2}{4}-b_0c_2$,  $b_{-1}=\frac{\theta _0}{2}(\frac{\theta _0}{2}-1)$ and 
 $a_2^2+c_2t-c_2^2(\theta _\infty +\frac{t^2}{4}-b_0c_2)-1=0$. The obvious morphism 
 $G\times ST_1\rightarrow \{(a_*,b_*,c_*)\in \mathcal{C}|\ c_1\neq 0\}$ is an isomorphism.\\
$\mathbb{C}[a_2,c_2,t,b_0]/ (a_2^2+c_2t-c_2^2(\theta _\infty +\frac{t^2}{4}-b_0c_2)-1)$ is the coordinate ring of
$ST_1$ and $ST_1$ is nonsingular.\\

 The {\it  `second standard form'}  $ST_2$ is the closed subset  of $\mathcal{C}$ defined by:\\
 $z\frac{d}{dz}+{a_0\ \ b\choose c_1z+z^2\ -a_0 }$ with $b=z^2+b_1z+b_0+b_{-1}z^{-1}$, $c=c_1z+z^2$ and 
  $c_1+b_1=t$, $b_1c_1+b_0=\theta _\infty +\frac{t^2}{4}$,  $a_0(a_0-1)+b_{-1}c_1=\frac{\theta_0}{2}(\frac{\theta_0}{2}-1)$.
The obvious morphism $G\times ST_2\rightarrow \{(a_*,b_*,c_*)\in \mathcal{C}|\ c_2\neq 0\}$ is an isomorphism.\\
$\mathbb{C}[a_0,c_1,t,b_{-1}]/(a_0(a_0-1)+b_{-1}c_1-\frac{\theta_0}{2}(\frac{\theta_0}{2}-1))$ is the coordinate
ring of $ST_2$.   {\it For fixed $t$, one finds one singular point}:  $a_0=1/2,\ b_{-1}=c_1=0,\ \theta_0=1$.\\
In the above case, one easily verifies that $D$ has the form 
$z\frac{d}{dz}+{\frac{1}{2}\ \ 0 \choose 0\ \ \frac{1}{2} }$ w.r.t. a basis of $\widehat{\mathcal{V}}_0$.
The quotient $\mathcal{C}/G$ is obtained by gluing the two `charts' $ST_1$ and $ST_2$ in the obvious way.   \end{proof}

\begin{observations} The level structure for $\mathcal{M}(\theta _0,\theta_\infty)$.\\
 {\rm For a connection $D\in \mathcal{M}(\theta _0,\theta_\infty)$, the level structure is a 1-dimensional
submodule $N\subset \mathbb{C}((z))\otimes \widehat{\mathcal{V}}_0$ with a generator $n$ such that
$\delta n=e^{\pi i \theta_0}n$. The space $\mathcal{M}^+(\theta _0,\theta_\infty)$ denotes the addition of this level structure to $\mathcal{M}(\theta _0,\theta_\infty)$. 

If $top_0$, the topological monodromy at $z=0$ of the connection $D$, is not $\pm {1\ 0\choose 0 \ 1 }$, then the level structure $N$ is unique. 

If $top_0=\pm {1\ 0\choose 0 \ 1 }$, then $\theta_0\in \mathbb{Z}$. Further $\widehat{\mathcal{V}}_0$
has a basis $v_1,v_2$  for which $D$ has the form
$z\frac{d}{dz}+{\frac{\theta_0}{2}\ \ \ \ 0\choose 0\ \ 1-\frac{\theta_0}{2}  }$. If
$\frac{\theta_0}{2}\neq 1-\frac{\theta_0}{2}$, the basis $v_1,v_2$ is unique up to multiplication by constants. Then one defines the level structure $N$ by $N=\mathbb{C}((z))v_1$.

In the final case $top_0=-{1\ 0\choose 0\ 1 }$ and $\theta _0=1$, the connection $D$ does not prescribe a level structure. We replace  $\mathcal{M}(1,\theta_\infty)$ by
 $\mathcal{M}^+(1,\theta_\infty)$ defined as the closed subspace of $\mathcal{M}(1,\theta_\infty)\times \mathbb{P}^1$ consisting of the equivalence classes of the tuples $(D,L)$ with
  $D\in \mathcal{M}(1,\theta_\infty)$ and $L$ a line in $\widehat{\mathcal{V}}_0$ at $z=0$, invariant under $D$. We will verify that  (for fixed $t$) $\mathcal{M}^+(1,\theta_\infty)\rightarrow \mathcal{M}(1,\theta_\infty)$ is the minimal resolution.
\smallskip

\noindent 
{\it Verification}. The chart $ST_2$ of $\mathcal{M}(1,\theta_\infty)$ consists of the differential
operators $z\frac{d}{dz}+{a_0\ \ \ zb\choose c_1+z\ \ \ 1-a_0}$ with 
$b=z^2+b_1z+b_0+b_{-1}z^{-1},\ c_1+b_1=t,\ b_1c_1+b_0=\theta_\infty +\frac{t^2}{4}, \ a_0(a_0-1)+b_{-1}c_1
=-\frac{1}{4}.$
The line $L= \mathbb{C}{x_1\choose x_0 }$ is generated by a nonzero element ${x_1\choose x_0 }\in \mathbb{C}[[z]]^2$  satisfying  $\{z\frac{d}{dz}+{a_0\ \ \ zb\choose c_1+z\ \ \ 1-a_0}\}{x_1\choose x_0 }=\frac{1}{2}{x_1\choose x_0}$. In the case $a_0=\frac{1}{2}, \ b_{-1}=c_1=0$, the operator  $z\frac{d}{dz}+{a_0\ \ \ zb\choose c_1+z\ \ \ 1-a_0}$ is equivalent over $\mathbb{C}[[z]]$ to $z\frac{d}{dz}+{\frac{1}{2}\ 0\choose 0\ \frac{1}{2} }$. Thus the possible lines
 $L$ form a projective line. In the opposite case, the operator is equivalent over $\mathbb{C}[[z]]$ to 
$z\frac{d}{dz}+{\frac{1}{2}\ 1\choose 0\ \frac{1}{2} }$ and there is only one $L$.
}\end{observations}
 
 \begin{observations} The reducible locus of  $\mathcal{M}(\theta _0,\theta_\infty)$.\\ {\rm
Put $\omega =z^2+\frac{t}{2}z+\frac{\theta_\infty}{2}$ and let $c\in \mathbb{C}[z^{-1},z]$. If
a reducible connection is present in $\mathcal{M}(\theta_0,\theta_\infty)$, then
$\frac{\theta_0}{2}\in \pm \frac{\theta_\infty}{2}+\mathbb{Z}$.
There are two types of reducible modules in ${\bf S}^+$. 
 { Type} (1) is represented by $z\frac{d}{dz}+{-\omega\  0\choose c\ \ \omega }$ and 
 { Type} (2) is represented by $z\frac{d}{dz}+{\omega \ \ 0\choose c\ -\omega}$. \\
 
 \noindent 
 For a given reducible module $M$, say of type (1), one defines (as before) the connection $(\mathcal{V},D)$ with generic  fibre $M$ by the local operators $z\frac{d}{dz}+{\frac{\theta_0}{2}\ \ \ \ *\choose 0\ \ 1-\frac{\theta_0}{2}  }$ at $z=0$ and $z\frac{d}{dz}+{\omega \ \ 0\choose 0\ -\omega }$ at $z=\infty$.
 Assume that type (1) is not present in  $\mathcal{M}(\theta _0,\theta_\infty)$. Then
 $\mathcal{V}\cong O(k)\oplus O(-k-1)$ with $k\geq 1$ and one identifies $\mathcal{V}$
 with $O(k[0])e_1\oplus O(-(k+1)[0])e_2$. A computation of $D$ in this case leads to two possible relations, namely $\frac{\theta_0}{2}=\frac{\theta_\infty}{2}-k$ or $1-\frac{\theta_0}{2}=\frac{\theta_\infty}{2}-k$. Thus one finds the list for type (1). The list for type (2) is found in a similar way. \\
  
 \noindent Type (1) is present in precisely the following cases:\\
 $\theta _0\geq \theta _\infty $ for  $\frac{\theta_0}{2}\in \frac{\theta _\infty}{2}+ \mathbb{Z}$ and  
 $\frac{\theta_0}{2}\not \in -\frac{\theta _\infty}{2}+ \mathbb{Z}$\\ 
  $\theta _0\leq -\theta _\infty +2 $ for $\frac{\theta_0}{2}\not \in \frac{\theta _\infty}{2}+ \mathbb{Z}$ and 
 $\frac{\theta_0}{2}\in -\frac{\theta _\infty}{2}+ \mathbb{Z}$\\ 
  $\theta _0\geq \theta _\infty $ or $\theta _0\leq -\theta _\infty +2 $ for  $\frac{\theta_0}{2} \in \frac{\theta _\infty}{2}+ \mathbb{Z}$ and $\frac{\theta_0}{2}\in -\frac{\theta _\infty}{2}+ \mathbb{Z}$\\

 \noindent Type (2) is present in precisely the following cases:\\
 $\theta _0\leq \theta _\infty +2 $ for  $\frac{\theta_0}{2}\in \frac{\theta _\infty}{2}+ \mathbb{Z}$ and 
 $\frac{\theta_0}{2}\not \in -\frac{\theta _\infty}{2}+ \mathbb{Z}$\\ 
  $\theta _0\geq -\theta _\infty $ for $\frac{\theta_0}{2}\not \in \frac{\theta _\infty}{2}+ \mathbb{Z}$ and 
 $\frac{\theta_0}{2}\in -\frac{\theta _\infty}{2}+ \mathbb{Z}$\\ 
  $\theta _0\leq \theta _\infty +2 $ or $\theta _0\geq -\theta _\infty $ for  $\frac{\theta_0}{2} \in \frac{\theta _\infty}{2}+ \mathbb{Z}$ and $\frac{\theta_0}{2}\in -\frac{\theta _\infty}{2}+ \mathbb{Z}$\\

   \noindent  {\it Examples}: We use the notation $z\frac{d}{dz}+{a\ \ b\choose c\ -a}$ of the space $\mathcal{C}$. Suppose $b=0$.\\
Then  $ a_2^2=1,\ 2a_1a_2=t,\ 2a_0a_2+a_1^2=\theta _\infty +\frac{t^2}{4},\ (a_0-\frac{1}{2})^2=(\frac{\theta _0}{2}-\frac{1}{2})^2$
and the non zero element $c_1z+c_2z^2$ is unique up to multiplication by a non zero constant. One finds in general four reducible families (with some overlap for $\theta _0=1$ and/or $\theta _\infty =\pm 1$):\\
$b=0,\ \ a_2=1,\ \ \ a_1=\frac{t}{2},\ \ \ a_0=\frac{\theta_\infty}{2},\ \ \ \frac{\theta_\infty}{2}=\frac{1}{2}\pm (\frac{\theta_0}{2}-\frac{1}{2})$ and \\
$b=0,\ \ a_2=-1,\ a_1=-\frac{t}{2},\ a_0=-\frac{\theta_\infty}{2},\  -\frac{\theta_\infty}{2}=\frac{1}{2}\pm (\frac{\theta_0}{2}-\frac{1}{2})$.\\
We observe that these examples are precisely the cases of an equality sign in 
the lists for type (1) and type (2).}\end{observations}

\begin{proposition} Put  $\beta =e^{\pi i\theta _0},\ \alpha =e^{\pi i \theta_\infty}$.
Let $F:\mathcal{M}^+(\theta _0,\theta _\infty)\rightarrow  {\bf S}^{+}(\beta ,\alpha )$ be the map
that sends a tuple $(D,L)$ to $(M,N)$ where $M$ is the generic fibre of $D$ and $N=\mathbb{C}((z))\otimes L$. The map $F$ is injective and its image contains the `irreducible locus' of ${\bf S}^{+}(\beta ,\alpha )$.  A component of the `reducible locus' lies in the image of $F$ if and only if 
 $\theta_0, \theta _\infty$ satisfy the corresponding inequality of Observations {\rm 2.3}.
\end{proposition}
\begin{proof} The injectivity of $F$ follows from the construction of $\mathcal{M}^+(\theta_0,\theta_\infty)$. If $M$ is irreducible then the vector bundle $\mathcal{W}$, introduced in the beginning of this section,
has type $O(k)\oplus O(-k)$ with $k\in \{0,1\}$. Therefore the subbundle $\mathcal{V}$ has type
$O\oplus O(-1)$ and can be identifies with $Oe_1\oplus O(-[0])e_2$. Thus the image of $F$ contains
the `irreducible locus' of ${\bf S}^+(\beta ,\alpha)$. The final statement follows from Observations 2.3. 
\end{proof}

{\it Define ${\bf S}^+(\theta_0,\theta_\infty)\subset {\bf S}^+(\beta ,\alpha)$ (for $\beta =e^{\pi i \theta_0}, \alpha =e^{\pi i \theta_\infty}$) to be the image of $F$ and let $\mathcal{R}^+(\theta_0,\theta_\infty)\subset \mathcal{R}^+(\beta ,\alpha)$ be the corresponding open subset.}

 \begin{corollary} $RH^{ext}:\mathcal{M}^+(\theta_0,\theta_\infty)\rightarrow \mathcal{R}^+(\theta_0,\theta_\infty)\times T$, the extended Riemann--Hilbert map,  is an analytic isomorphism.
  \end{corollary}
\begin{proof} $RH^{ext}$ is bijective since ${\bf S}^+(\theta_0,\theta_\infty)\rightarrow \mathcal{R}^+(\theta_0,\theta_\infty)\times T$ is bijective. The two spaces $\mathcal{M}^+(\theta_0,\theta_\infty)$
 and  $\mathcal{R}^+(\theta_0,\theta_\infty)\times T$ are smooth and so $RH^{ext}$ is an analytic
 isomorphism (see \cite{vdP1,vdP2}). \end{proof}

\subsection{Isomonodromic families, Okamoto--Painlev\'e spaces}
 {\it An isomonodromic family above the chart $ST_2$ of $\mathcal{M}^+(\theta_0, \theta_\infty)$}  has the form $z\frac{d}{dz}+A$ with $A={a_0\ \ \ b\choose c\ \ -a_0 }$ with $c=z^2-qz$, $b=z^2+b_1z+b_0+b_{-1}z^{-1}$,
$b=z^2+(t+q)z+q(t+q)+\theta _\infty +\frac{t^2}{4}+(\frac{(a_0-\frac{1}{2})^2-(\frac{\theta_0}{2}-\frac{1}{2})^2}{q})z^{-1}$, where $a_0$ and $q$ are functions of $t$. Isomonodromy is equivalent to the existence
of an operator $\frac{d}{dt}+B$, commuting with $z\frac{d}{dz}+A$. In other terminology
$z\frac{d}{dz}+A, \frac{d}{dt}+B$ is a {\it Lax pair}. This is equivalent to the equation 
$\frac{d}{dt}(A)=z\frac{d}{dz}(B)+[A,B]$.  One observes that $B$
has trace zero and that the entries of $B$ have the form $d_{-1}(t)z+d_0(t)+d_1(t)z$. Using MAPLE
one obtains the solution $B={0\ B_1\choose B_2\ 0 }$ with
$B_1=\frac{q(q+b_1)+b_0}{2}z^{-1}+\frac{b_1+q}{2}+\frac{1}{2}z,\ B_2=\frac{1}{2}z$,  
\[q'=a_0-\frac{1}{2},\ a_0'=\frac{b_{-1}+q(q(b_1+q)+b_0)}{2} \mbox{ and the fourth Painlev\'e equation } \] 
\[q''=\frac{(q')^2}{2q}+\frac{3q^3}{2}+tq^2+\frac{q}{8}(4\theta_\infty +t^2)-\frac{(\theta_0-1)^2}{8q}
\mbox{ with parameters } \theta_0,\theta_\infty .\]

\noindent {\it Isomonodromy for reducible families}.  
An isomonodromic  family of operators $z\frac{d}{dz}+{\omega \ \ 0\choose z^2-qz\ \ -\omega }$, with $\omega =z^2+\frac{t}{2}z+\frac{d}{2}$, commutes with an operator of the form
$\frac{d}{dt}+{\tau \ \ \  0\choose \frac{z}{2}\ -\tau }$. One computes that $\tau =\frac{2z+2q+t}{4}$
and  $q'=q^2+\frac{t}{2}q+\frac{d-1}{2}$. Then $q$ is a Riccati solution of ${\rm PIV}$ with $d=\theta_\infty$ and $d=1\pm (\theta_0-1)$.\\

 \noindent   
An isomonodromic family $z\frac{d}{dz}+{-\omega \ \ 0\choose z^2-qz\ \ \omega }$ with $\omega =z^2+\frac{t}{2}z+\frac{d}{2}$ produces the equation $q'=-q^2-\frac{t}{2}q-\frac{d+1}{2}$. Then $q$ is a Riccati solution of
${\rm PIV}$ with $d=\theta _\infty$ and $d=-1\pm (\theta _0-1)$.

\begin{observations} The solutions $q_r$ with  $r\in \mathcal{R}^+(\theta_0,\theta_\infty)$.\\
{\rm The fibre of $\mathcal{M}^+(\theta_0,\theta_\infty)\rightarrow  \mathcal{R}^+(\theta_0,\theta_\infty)$ above $r$ is, by Corollary 2.5, isomorphic to $T$. Write $q_r$ for the function $q$ appearing in the formula for
the chart $ST_2$. Then $q_r$ is a meromorphic solution of PIV, defined on all of $T$.

}\end{observations}

\begin{theorem} The fourth Painlev\'e equation has the Painlev\'e property. The moduli space
$\mathcal{M}^+(\theta_0,\theta_\infty)$ is analytically isomorphic to  the Okamoto--Painlev\'e space for 
${\rm PIV}$ with parameters $\theta_0, \theta_\infty$.\end{theorem}
\begin{proof} Let a local solution $Q$ of PIV with parameters $\theta_0,\theta_\infty$ be given.
Let $U$ be an open disk, where $Q$ is holomorphic and has no zeros.
 Consider the operator 
$z\frac{d}{dz}+{\tilde{a}_0\ \ \ \tilde{b}  \choose z^2-Qz\ -\tilde{a}_0}$ with $\tilde{a}_0=\frac{dQ}{dt}+\frac{1}{2}$,
$\tilde{b}=z^2+(t+Q)z+Q(t+Q)+\theta_\infty +\frac{t^2}{4}+
\frac{(\tilde{a}_0-\frac{1}{2})^2-(\frac{\theta_0}{2}-\frac{1}{2})^2}{Q}z^{-1}$. This defines an analytic
map $U\rightarrow \mathcal{M}^+(\theta_0,\theta_\infty)$. Since $Q$ is a local solution of PIV,
the map $U\rightarrow \mathcal{M}^+(\theta_0,\theta_\infty)\rightarrow \mathcal{R}^+(\theta_0,\theta_\infty)$ is constant. Let $r$ be its image. Then $Q$ coincides with $q_r$ on $U$. Thus $Q$ extends
to a global solution of PIV and this equation has the Painlev\'e property. 

The bundle $\mathcal{M}^+(\theta_0,\theta_\infty)\rightarrow T$, with its foliation defined by the fibres
of $\mathcal{M}^+(\theta_0,\theta_\infty)\rightarrow \mathcal{R}^+(\theta_0,\theta_\infty)$, is the 
Okamoto--Painlev\'e variety according to the isomorphism of Corollary 2.5. \end{proof}
We note that
$ \mathcal{R}^+(\theta_0,\theta_\infty)$ is the {\it space of initial conditions}.

\subsection{$Aut({\bf S}^{+})$ and  B\"acklund transformations}
Natural automorphisms of  ${\bf S}^{+}$ are:\\
(1). $\sigma _1:(M,N)\mapsto (M,N^*)$ where $N^*$ is a submodule of  $\mathbb{C}((z))\otimes M$ such that $N\oplus N^*=\mathbb{C}((z))\otimes M$. This is well defined for $\beta \neq \pm 1$. 
For $\beta =\pm 1$, the module $N^*$ might not exist or might not be unique.  It seems correct to
define $N^*:=N$ for $\beta =\pm 1$. \\
(2). $\sigma _2:(M,N)\mapsto (M\otimes A,N\otimes A)$, where $A=\mathbb{C}((z))a$ and $\delta a=\frac{1}{2}a$.\\
(3). $\sigma _3:(M,N)\mapsto \mathbb{C}(z)\otimes _\phi (M,N)$, where $\phi$ is the automorphism of 
$\mathbb{C}(z)$ which is the identity on $\mathbb{C}$ and maps $z$ to $iz$. Let $Aut({\bf S}^{+})$ denote the group generated by $\sigma _j,\ j=1,2,3$.
\begin{center}
\begin{tabular}{|c|c|c|c|c|}
\hline
                     & $\beta$ & $\alpha $ & $t$ & $z$ \\
\hline
$\sigma _1$&  $\beta ^{-1}$ & $\alpha $ & $t$ & $z$ \\
\hline
$\sigma _2$ &$-\beta$ & $-\alpha $ & $t$ & $z$ \\
\hline
$\sigma _3$ & $\beta$ & $\alpha ^{-1} $ & $it$ & $iz$ \\
\hline
\end{tabular}
\end{center}
The above group is commutative and has order 16. The following table is a choice of lifting the generators to actions on $\theta_0,\theta_\infty, t,z$.
\begin{center}
\begin{tabular}{|c|c|c|c|c|}
\hline
                     & $\theta _0$ & $\theta _\infty $ & $t$ & $z$ \\
\hline
$\tilde{\sigma} _1$&  $2-\theta_0 $ & $\theta_\infty $ & $t$ & $z$ \\
\hline
$\tilde{\sigma}_2$ &$\theta_0 +1$ & $\theta_\infty +1 $ & $t$ & $z$ \\
\hline
$\tilde{\sigma}_3$ & $\theta _0$ & $-\theta_\infty $ & $it$ & $iz$ \\
\hline

\end{tabular}
\end{center}

The induced morphisms  $\tilde{\sigma} _1:\mathcal{M}^+(\theta _0,\theta_\infty)
\rightarrow \mathcal{M}^+(-\theta_0 +2,\theta _\infty )$ and
 $\tilde{\sigma} _3:\mathcal{M}^+(\theta _0,\theta_\infty) \rightarrow \mathcal{M}^+(\theta_0,-\theta _\infty )$ are evident from the standard operators representing the points of
  $\mathcal{M}^+(\theta _0,\theta_\infty)$.
A MAPLE computation yields the explicit morphisms $\tilde{\sigma} _2:\mathcal{M}^+(\theta _0,\theta_\infty) \rightarrow \mathcal{M}^+(\theta_0 +1,\theta _\infty +1)$. 
The formulas are given with respect to the coordinates $a=a_0,q$ of an open subset (namely $q\neq 0$ in the chart $ST_2$) of the first space and
 $\tilde{a}=\tilde{a}_0,\tilde{q}$ of an open subset of the second space.
The assumption that the operator $z\frac{d}{dz}+A(a,q,\theta_0, \theta _\infty,z)$, belonging to
$\mathcal{M}^+(\theta _0,\theta_\infty)$, is equivalent, by a transformation of the type $U_{-2}z^{-2}+U_{-1}z^{-1}+U_0+U_1z+U_2z^2$, to the operator
 $z\frac{d}{dz}+A(\tilde{a},\tilde{q},\theta_0 +1, \theta _\infty +1,z)$, belonging to   $\mathcal{M}^+(\theta_0 +1,\theta _\infty +1)$, leads to the following formulas
\[\tilde{q}=\frac{-4q^2\theta_\infty +4a^2-4q^3t-q^2t^2-4q^4-4q^2\theta_0+\theta_0^2-4a\theta_0}
{4q(qt-\theta_0+2a+2q^2)} \]
\[\tilde{a}=\frac{long}{16q^2(qt-\theta _0+2a+2q^2)^2}.\]
The substitution $a=q'+\frac{1}{2}$ in the first formula produces $\tilde{q}$ in terms of $q,q'$ and the parameters $\theta_0,\theta_\infty$, this is {\it the B\"acklund transformation in terms of solutions}.
 The second formula is obtained by substitution $\tilde{a}=\tilde{q}'+\frac{1}{2}$
and an expression for $\tilde{q}'$ coming from the first formula and the equation for $q''$. \\

The term $qt-\theta_0+2a+2q^2$ 
in the denominator of the formulas indicates that the morphism $\tilde{\sigma }_3$ is in
general a rational equivalence and is not defined  on leaves of the foliation with $a=q'+\frac{1}{2}$ and
$q'+q^2+\frac{t}{2}q+\frac{-\theta_0+1}{2}=0$. This occurs precisely when  
$\theta_0=-\theta_\infty$  and the reducible locus of $\mathcal{M}^+(\theta_0,\theta_\infty)$
is not present in the corresponding $\mathcal{M}^+(1+\theta_0,1+\theta_\infty)$ (compare Observations 2.3 and the Riccati equations for reducible families).

\bigskip

We note that the  group $<\tilde{\sigma} _1,\tilde{\sigma }_2,\tilde{\sigma} _3>$ contains the two shifts 
$\theta_0\mapsto \theta _0+2,\  \theta_\infty \mapsto \theta _\infty$ and 
$\theta_0\mapsto \theta_0,\  \theta _\infty \mapsto \theta _\infty +2$. 
One observes that, in comparison  with the book of Gromak et al. \cite{Gr} and Okamoto's paper \cite{O3}, there is still a missing generator for the group of all B\"acklund transformations of PIV. This generator does
not seem to come from a `natural' automorphism of ${\bf S}^+$ (i.e., constructions of linear algebra for differential modules and operations with  the differential field $\mathbb{C}(z)$ ).   In the final section we will investigate another set of differential modules $\bf S$, inspired by Noumi's symmetric form of PIV (\cite{No,No-Y}). As is shown by Noumi, this will easily produce all B\"acklund transformations and moreover all rational and Riccati solutions.

\section{The Noumi--Yamada family}
M.~Noumi  and Y.~Yamada produced a $3\times 3$-Lax pair, arising from the Lax formalism of the modified KP hierarchy,  for the symmetric form of ${\rm PIV}$, namely
\[z\frac{d}{dz}+\left(\begin{array}{ccc}\epsilon _1 &f_1 & 1 \\ z&\epsilon _2 & f_2\\ f_0z&z &\epsilon _3 \end{array}\right),\
\frac{d}{dt} +\left(\begin{array}{ccc} -q_1 &1 & 0 \\ 0 &-q_2 &1 \\ z &0 &-q_3 \end{array}\right),
\mbox{ leading to equations}  \]
 \begin{small}
\[\epsilon _1'=\epsilon_2'=\epsilon_3'=0; \ f_1-f_2=-q_1+q_3;\ f_2-f_0=q_1-q_2;\ f_0-f_1=q_2-q_3,\]
\[f_0'=f_0(f_1-f_2)+(1-\epsilon_1+\epsilon_3);
\ f_1'=f_1(f_2-f_0)+(\epsilon _1-\epsilon _2);\
f_2'=f_2(f_0-f_1)+(\epsilon _2-\epsilon _3).\]
\end{small}
Since the local exponents $\epsilon _*$ at $z=0$ are constants in an isomonodromic  family, 
we can and will suppose $\epsilon_1+\epsilon _2+\epsilon _3=0$. Then we may and will also suppose that $q_1+q_2+q_3=0$. Further $t=f_0+f_1+f_2$ is assumed.

Then $f_1$ satisfies $y''=\frac{(y')^2}{2y}+\frac{3}{2}y^3-2ty^2+(\frac{t^2}{2}+\theta_\infty)y-
\frac{(\theta_0-1)^2}{2y}$ with $\theta_0=1+\epsilon_1-\epsilon _2$, $\theta_\infty=1+\epsilon _1-\epsilon_3$. After rescaling $t\mapsto \frac{t}{ \sqrt{2}},\ y\mapsto -\sqrt{2}y$ one obtains   `our' equation
  $y''=\frac{(y')^2}{2y}+\frac{3}{2}y^3+ty^2+(t^2+4\theta_\infty )\frac{y}{8}-\frac{(\theta_0-1)^2}{8y}$.

 Using this symmetric form one finds the extended Weyl group of $A_2$ as group of B\"acklund transformations.
For example $\pi :(f_0,f_1,f_2)\mapsto (f_1,f_2,f_0)$ translates into the `missing' B\"acklund transformation
of \S 2.3 , namely \\
$\theta _0\mapsto -\frac{1}{2}\theta _0+\frac{1}{2}\theta _\infty +2,\ \ \ \theta _\infty \mapsto -\frac{3}{2}\theta_0-\frac{1}{2}\theta_\infty +2.$\\

This is the inspiration for the new class $\bf S$ of differential modules $M$ over $\mathbb{C}(z)$, defined by:
$\dim M=3$; $\Lambda ^3M$ is trivial; the only singular points are $0,\infty$; $0$ is regular singular and the Katz
invariant of $\infty$ is $\frac{2}{3}$. After scaling the variable $z$ the generalized eigenvalues at $\infty$ are:\\  $q_0=z^{2/3}+\frac{t}{3}z^{1/3},\ q_1=\zeta ^2z^{2/3}+\zeta \frac{t}{3}z^{1/3},\ 
q_2=\zeta z^{2/3}+\zeta ^2\frac{t}{3}z^{1/3}$ where $\zeta =e^{2\pi i/3}$.\\

\noindent {\it Invariant lattices at $z=\infty$}.
For $M\in {\bf S}$ the operator $D=\nabla _{z\frac{d}{dz}}$ has at $z=\infty$ has the form 
$z\frac{d}{dz}+diag(q_0,q_1,q_2)$ with respect to a basis $e_0,e_1,e_2$. A lattice $\Lambda$ at $z=\infty$ is called {\it  invariant} if $z^{-1}D(\Lambda )\subset \Lambda$. We respect to the basis 
$h_0=e_0+e_1+e_2,\ h_1=z^{1/3}(e_0+\zeta e_1+\zeta ^2e_2),\ h_2=z^{-1/3}(e_0+\zeta ^2e_1+\zeta e_2)$, 
$D$ has the form $z\frac{d}{dz}+\left(\begin{array}{ccc}    0 & z & \frac{t}{3} \\ \frac{t}{3} & \frac{1}{3}& 1 \\ z & \frac{t}{3}z& -\frac{1}{3}\end{array}\right)$.
Thus $\Lambda _0:=<h_0,h_1,h_2>$ is an invariant lattice. $\Lambda _1:=<h_0,z^{-1}h_1,h_2>$ 
is the only invariant lattice of codimension 1 in $\Lambda_0$ and the operator $D$ has with respect to this basis the form
$z\frac{d}{dz}+\left(\begin{array}{ccc} 0&1 &\frac{t}{3} \\ \frac{t}{3}z & -\frac{2}{3}& z \\ z &\frac{t}{3} &-\frac{1}{3} \end{array}\right)$.
The invariant lattice $\Lambda _2:=<z^{-1}h_0,z^{-1}h_1,h_2>$ has codimension 2 in $\Lambda _0$.
 All invariant lattices are  $\{z^n\Lambda _i| \  n\in \mathbb{Z},\ i= 0,1,2\}$.  \\

 The Noumi--Yamada Lax pair has one additional feature, namely:\\
 there exists  $U\in {\rm GL}(3,\mathbb{C}[[z]])$ with $U=1+U_1z+U_2z^2+\dots$ such that
 \[U^{-1}\{ z\frac{d}{dz}+\left(\begin{array}{ccc}\epsilon _1 &f_1 & 1 \\ z&\epsilon _2 & f_2\\ f_0z&z &\epsilon _3 \end{array}\right)\} U=z\frac{d}{dz}+
 \left(\begin{array}{ccc}\epsilon _1 &* & * \\ 0&\epsilon _2 & *\\ 0 & 0&\epsilon _3 \end{array}\right),
\mbox{ with  all }*\in \mathbb{C}.\]

\noindent {\it Level structure for $\bf S$}. This leads to a `level structure' or `parabolic structure' for the elements $M\in {\bf S}$ consisting of differential submodules  $M_1\subset M_2\subset \mathbb{C}((z))\otimes M$ of dimensions 1 and 2 over $\mathbb{C}((z))$. Let ${\bf S}^+$ denote the set of the differential modules in $\bf S$, provided with a level structure.  \\

\noindent {\it The moduli space $\mathcal{R}$ for the analytic data}.\\ 
 The singular directions $d$ for $q_k-q_\ell$ are computed as follows. $z\frac{d}{dz}(y)=(q_k-q_\ell)y$ has solution
 $\exp (\frac{3}{2}(\zeta ^{2k}-\zeta ^{2\ell})z^{2/3}+3(\zeta ^k-\zeta ^\ell)z^{1/3} )$. Write $z=e^{id}$ and 
 $\zeta ^{2k}-\zeta ^{2\ell} =|\zeta ^{2k}-\zeta ^{2\ell}|e^{i\phi(k,\ell)}$. Then $|y(re^{id})|$ has maximal descent for $r\rightarrow \infty$ if and only if $\frac{2}{3}d+\phi (k,\ell )=\pi +\mathbb{Z}2\pi$ (or $d=\frac{3}{2}\pi -\frac{3}{2}\phi (k,\ell)+\mathbb{Z}3\pi$).
\begin{small} \begin{center}
\begin{tabular}{|c|c|c|c|}
\hline
$k$&$\ell$&$\phi$&$d$\\
\hline
$0 $& $ 1$& $\frac{1}{6} \pi $ & $\frac{5}{4} \pi $ \\
\hline
$ 1$& $0 $& $\frac{7}{6} \pi $ & $\frac{11}{4} \pi$ \\
\hline
$ 0$& $2 $& $\frac{11}{6}\pi $ & $\frac{7}{4}\pi $ \\
\hline
$ 2$& $0 $& $\frac{5}{6}\pi $ & $\frac{1}{4}\pi $ \\
\hline
$ 1$& $2 $& $\frac{9}{6}\pi $ & $\frac{9}{4}\pi $ \\
\hline
$ 2$& $1 $& $\frac{3}{6}\pi $ & $\frac{3}{4}\pi $ \\
\hline

\end{tabular}
\end{center} 
\end{small}
The analytic data consists of the formal monodromy and the six Stokes matrices at $z=\infty$.
The product of the formal monodromy and the Stokes matrices  for the singular directions $d\in [0,2\pi )$
\[\left(\begin{array}{ccc} 0&0 & 1 \\ 1 &0 &0 \\ 0 &1 & 0\\ \end{array}\right)
\left(\begin{array}{ccc}1 &0 &0 \\ 0&1 &0 \\ x_4&0 &1 \\ \end{array}\right)
\left(\begin{array}{ccc} 1&0 & 0\\  x_3&1 &0 \\ 0 & 0& 1\\ \end{array}\right)
\left(\begin{array}{ccc} 1 &0 & 0\\ 0 & 1 & x_2\\ 0 & 0& 1\\ \end{array}\right)
\left(\begin{array}{ccc} 1 &0 & x_1 \\ 0 & 1& 0\\ 0 &0 &1 \\ \end{array}\right) \]
is equal to the topological monodromy $M(x_*):=\left(\begin{array}{ccc} x_4&0 & x_1x_4+1 \\ 1 &0 &x_1 \\ x_3 &1 & x_1x_3+x_2  \end{array}\right)$.
Its characteristic polynomial is
 $\lambda ^3-(x_2+x_4+x_1x_3)\lambda ^2+(-x_1-x_3+x_2x_4)\lambda -1.$\\
 \indent {\it The  moduli space $\mathcal{R}$} for the analytic data consists of the tuples $x_*=(x_1,\dots ,x_4)$ since the other two Stokes matrices can be expressed in the Stokes matrices for the singular directions
in $[0,2\pi )$.  Thus $\mathcal{R}\cong \mathbb{A}^4$. The elements of {\it the parameter space $\mathcal{P}$} are the sets of eigenvalues of the topological monodromy. Thus 
$\mathcal{P}=\{\lambda ^3-e_1\lambda ^2+e_2\lambda -1|\ e_1,e_2\in \mathbb{C}\}$. Let $\mathcal{R}(P)$ be
the fibre above $P\in \mathcal{P}$. If $P$ has three distinct roots, then $\mathcal{R}(P)$
is a smooth surface. If $P$ has roots $a,a,a^{-2}$ with $a\neq a^{-2}$, then the point  $x_*\in \mathcal{R}(P)$ with
$x_*\neq (-a^{-1} ,a ,-a^{-1} ,a )$ is regular and $M(x_*)$ has two Jordan blocks. The point 
$(-a^{-1} ,a ,-a^{-1} ,a )\in \mathcal{R}(P)$ is singular and has type $A_1$. Further 
$M(-a^{-1} ,a ,-a^{-1} ,a )$ has three Jordan blocks. If  $P$ has roots $a,a,a$ (and thus $a^3=1$), then the point
 $x_*\in \mathcal{R}(P)$ with $x_*\neq (-a^{-1} ,a ,-a^{-1} ,a )$ is regular and $M(x_*)$ has one Jordan block.
 The point $x_*=(-a^{-1} ,a ,-a^{-1} ,a )$ is singular and has type $A_2$. The matrix $M(-a^{-1} ,a ,-a^{-1} ,a )$ has
 two Jordan blocks. \\

 \noindent {\it Level structure for $\mathcal{R}$ and $\mathcal{P}$}. For an element of 
 $\mathcal{R}$ we introduce a `level structure' which consists of subspaces
 $L_1\subset L_2\subset \mathbb{C}^3$ of dimensions 1 and 2 which are invariant under the topological monodromy (at $z=\infty$ or, equivalently, at $z=0$). The corresponding space is  denoted
 by $\mathcal{R}^+$. The level structure for a $P\in \mathcal{P}$ consists of a tuple
 $(\mu_1,\mu_2,\mu_3)$ with $\mu_1\mu_2\mu_3=1$ and 
 $P=\prod _{j=1}^3(\lambda-\mu_j)$. The corresponding space is denoted by $\mathcal{P}^+$. 
 The morphism $par:\mathcal{R}^+\rightarrow \mathcal{P}^+$ is defined by $((x_*),L_1,L_2)\mapsto
 (\mu_1,\mu_2,\mu_3)$, where $\mu_1$ is the eigenvalue of $M(x_*)$ on $L_1$ and $\mu_2$ is that of  $M(x_*)$ on $L_2/L_1$. 
 
 \bigskip
 
One observes that $\mathcal{R}^+$ is the closed subspace of
$ \mathbb{C}^4_{x_*}\times \mathbb{P}^2_{y_*}\times (\mathbb{P}^2)^*_{z_*}\times \mathcal{P}^+$,
where $\mathcal{P}^+=\{(\mu _1 ,\mu _2,\mu _3)\in \mathbb{C}^3| \mu_1\mu_2 \mu _3=1\}$, given
by the equations:\\
$M(x_*)y=\mu_1y,\ y:=\left(\begin{array}{c} y_1\\ y_2\\ y_3\end{array}\right)$,
$zM(x_*)=\mu_3z,\ z:=(z_1,z_2,z_3)$, $\sum y_jz_j=0$.\\
Indeed, $\mathbb{C}y$ and the  kernel of $z\in (\mathbb{C}^3)^*$ are the $M(x_*)$-invariant spaces 
$L_1\subset L_2\subset \mathbb{C}^3$. Further $par:\mathcal{R}^+\rightarrow \mathcal{P}^+$ is the projection onto the last factor.
The fibre $\mathcal{R}^+(\mu_*)$ of $par$ above the point $(\mu_*)\in \mathcal{P}^+$ maps to the fibre
$\mathcal{R}(P)$ of $\mathcal{R}\rightarrow \mathcal{P}$ above the point
$P=(\lambda-\mu_1)(\lambda-\mu_2)(\lambda-\mu_3)$. 
\begin{proposition}
$res:\mathcal{R}^+(\mu_*)\rightarrow \mathcal{R}(P)$ is the minimal resolution of 
$\mathcal{R}(P)$. \end{proposition}
A straightforward computation proves this statement. In particular, the fibre of $res$ above a non singular point
is just one point since there is only one level structure possible. If two of the $\mu_*$ are equal, then the preimage  under $res$ of the singular point  is a $\mathbb{P}^1$, consisting of the lines 
$\mathbb{C}y$ in the two-dimensional eigenspace for $a$ and  the kernel of $z$ is this two-dimensional eigenspace.   If the three $\mu_*$ are equal, then  the preimage  under $res$ of the singular point is {\it a pair of intersecting projective lines}. In this case the Jordan form of $M(x_*)$ has two blocks, $\mathbb{C}y$ is a line in the two-dimensional eigenspace of $\mathbb{C}^3$, $\mathbb{C}z$ is a line in the two-dimensional eigenspace of the dual $(\mathbb{C}^3)^*$ and $\sum y_jz_j=0$.
\begin{proposition} The natural maps ${\bf S}\rightarrow \mathcal{R}\times T$ and 
${\bf S}^+\rightarrow \mathcal{R}^+\times T$ (with $T=\mathbb{C}$) are bijections. \end{proposition}
\begin{proof} In the first case one applies [vdP-Saito], Theorem 1.7. The second case follows from the observation that the level structure $M_1\subset M_2\subset \mathbb{C}((z))\otimes M$ induces subspaces $L_1\subset L_2\subset \mathbb{C}^3$ of dimensions 1 and 2, invariant under the topological monodromy, 
and visa versa. \end{proof}

 \noindent {\it The Noumi--Yamada moduli space $\mathcal{N}^+(\epsilon_*)$}.\\
 $\epsilon_*$ denotes a triple $(\epsilon_1,\epsilon_2,\epsilon_3)$ with $\sum \epsilon_j=0$.
The set ${\bf S}^+(\epsilon_*)$ consists of the tuples
 $(M, M_1\subset M_2)\in {\bf S}^+$ such that $M_1=\mathbb{C}((z))b_1$ with
 $\delta_M(b_1)= \epsilon_1b_1$ and $M_2/M_1=\mathbb{C}((z))b_2$ with
 $\delta_M(b_2)=\epsilon_2b_2$. {\it Let $\mathcal{V}$ denote the free bundle on $\mathbb{P}^1$ of rank 3}. 
 
 The points of the moduli space   $\mathcal{N}^+(\epsilon_*)$  are the isomorphy classes of connections  $D:=\nabla_{z\frac{d}{dz}}:\mathcal{V}\rightarrow O([\infty])\otimes  \mathcal{V}$ with a level structure consisting of 
 $D$-invariant submodules $V_1\subset V_2 \subset \widehat{\mathcal{V}}_0$ of rank 1 and 2 such that 
 $\widehat{\mathcal{V}}_0/V_j, \ j=1,2$ have no torsion and such that  there is a tuple $(M, M_1\subset M_2)\in {\bf S}^+(\epsilon_*)$ with $M$ is the generic fibre of $\mathcal{V}$, $\mathbb{C}((z))\otimes V_j=M_j,\ j=1,2$ and  
 $\widehat{\mathcal{V}}_\infty$ is the lattice $\Lambda _0\subset \mathbb{C}((z^{-1}))\otimes M$.
  
 \begin{proposition}  $\mathcal{N}^+(\epsilon_*)$ is the affine space $\mathbb{A}^3$ with coordinates
 $f_0,f_1,f_2$, $t=f_0+f_1+f_2$ and the connection is represented by
 $z\frac{d}{dz}+\left(\begin{array}{ccc}\epsilon _1 &f_1 & 1 \\ z&\epsilon _2 & f_2\\ f_0z&z &\epsilon _3 \end{array}\right)$.
 \end{proposition}
 \begin{proof} The level structure provides $H^0(\mathcal{V})$ with a basis $e_1,e_2,e_3$ such that 
 $D=z\frac{d}{dz}+A_0+A_1z$ with traceless constant matrices $A_0,A_1$ and 
  $A_0=\left(\begin{array}{ccc}\epsilon _1&*&*\\ 0&\epsilon _2&*\\ 0&0&\epsilon _3\end{array}\right)$.
This is unique up to the action of $B=\{\left(\begin{array}{ccc}*&*&*\\ 0&*&*\\ 0&0&*\end{array}\right)\}$.
 The lattice condition at $z=\infty$ is equivalent to
 $U\{z\frac{d}{dz}+A_0+A_1z\}U^{-1}=z\frac{d}{dz}+\left(\begin{array}{ccc} 0&z&\frac{t}{3}\\ \frac{t}{3}&\frac{1}{3}&1\\ z&\frac{t}{3}z&-\frac{1}{3}\end{array}\right)$ for some $U=U_0(1+U_{-1}z^{-1}+\dots )\in {\rm GL}_3(\mathbb{C}[[z^{-1}]])$.
This is equivalent to the equations 
$A_1= U_0^{-1}\left( \begin{array}{ccc} 0&1&0\\0&0&0\\1&\frac{t}{3}&0\end{array}\right)
   U_0$ and $ A_0=U_0^{-1} \left(\begin{array}{ccc} 0&0&\frac{t}{3}\\\frac{t}{3}&\frac{1}{3}&1\\ 0&0&-\frac{1}{3}\end{array}\right)  U_0 +[A_1,U_{-1}].$ A MAPLE computation produces matrices $U_0$ and $U_{-1}$ such that 
 $ A_0=\left(\begin{array}{ccc}\epsilon _1&*&*\\ 0&\epsilon _2&*\\
 0&0&\epsilon _3\end{array}\right)$  and $ A_1= \left(\begin{array}{ccc} 0&0&0\\ 1&0&0\\ *&1&0\end{array}\right) .$
 Moreover $U_0$ is unique up to 
multiplication by a non zero constant and $U_{-1}$ is unique up to adding a matrix $V$ with $[A_1,V]=0$.
Thus we found a representation of the connection in the `Noumi--Yamada form' and this form is unique 
with respect to the action of the Borel group $B$ on $\mathcal{V}$.
 Therefore the Noumi--Yamada form represents the moduli space  $\mathcal{N}^+(\epsilon_*)$. \end{proof}

The map $\mathcal{N}(\epsilon_*)\rightarrow {\bf S}^+(\epsilon_*)$ is injective and not bijective. This is due
to the choice of a free vector bundle $\mathcal{V}$ in the construction of  $\mathcal{N}(\epsilon_*)$.
The aim is to avoid this choice and to construct a smooth partial completion $\widehat{\mathcal{N}}(\epsilon_*)$ such that
 $\widehat{\mathcal{N}}(\epsilon_*)\rightarrow {\bf S}^+(\epsilon_*)$ is bijective. As in Section 2, this will imply
 that the extended Riemann--Hilbert map $\widehat{\mathcal{N}}(\epsilon_*)\rightarrow \mathcal{R}^+(\mu_*)\times T$ (with $\mu_j=e^{2\pi i \epsilon_j}$ for $j=1,2,3$) is an analytic isomorphism. Moreover
$\widehat{\mathcal{N}}(\epsilon_*)$ is the Okamoto--Painlev\'e space and $\mathcal{R}^+(\mu_*)$ is the space
of the initial conditions.\\

\noindent {\it Construction of  $\widehat{\mathcal{N}}(\epsilon_*)$}. \\
The points of $\widehat{\mathcal{N}}(\epsilon_*)$ are (the isomorphism classes of) the tuples 
$(\mathcal{V},D,L_1,L_2)$ with $D=\nabla_{z\frac{d}{dz}}:\mathcal{V}\rightarrow O([\infty])\otimes \mathcal{V}$
is a connection on a vector bundle $\mathcal{V}$ of rank 3. {\it We require the following}:
The connection $\widehat{\mathcal{V}}_\infty$ is isomorphic to $\Lambda_1$. In other terms
$\widehat{\mathcal{V}}_\infty$ has a basis over $\mathbb{C}[[z^{-1}]]$ for which $D$ has the form
$z\frac{d}{dz}+\left(\begin{array}{ccc} 0&1 &\frac{t}{3} \\ \frac{t}{3}z & -\frac{2}{3}& z \\ z &\frac{t}{3} &-\frac{1}{3} \end{array}\right)$. Further
$L_1=\mathbb{C}[[z]]Y$ is a subconnection of $\widehat{\mathcal{V}}_0$ such that $\widehat{\mathcal{V}}_0/L_1$ has no torsion and $DY=\epsilon _1Y$. Further
$L_2=\mathbb{C}[[z]]Z$ is a subconnection of $\widehat{\mathcal{V}}_0^*$, the dual of $\widehat{\mathcal{V}}_0$, such that $\widehat{\mathcal{V}}_0^*/L_2$ has no torsion
and $DZ=\epsilon _3Z$. Moreover $\Lambda ^3(\widehat{\mathcal{V}}_0)$ is trivial and  $L_2(L_1)=0$.\\

The map  $F:\widehat{\mathcal{N}}(\epsilon_*)\rightarrow {\bf S}^+(\epsilon_*)$,  sends 
$(\mathcal{V},D,L_1,L_2)$ to its generic fibre $M$ together with the level structure on $\mathbb{C}((z))\otimes M$ obtained from $L_1,L_2$. Conversely, for a given element $(M, M_1\subset M_2)\in {\bf S}^+(\epsilon_*)$
one defines the connection  $(\mathcal{V},D)$ with generic fibre $M$, by prescribing 
$\widehat{\mathcal{V}}_\infty \cong \Lambda _1$. The additional data $L_1,L_2$ imply that $\widehat{\mathcal{V}}_0$ is represented by 
\begin{small}$z\frac{d}{dz}+\left(\begin{array}{ccc}   \epsilon_1& *&* \\ 0 &\epsilon _2 &* \\ 0  & 0 & \epsilon _3\end{array}\right)$\end{small} with all $*\in \mathbb{C}[[z]]$. This implies that
{\it $F$ is bijective in the following cases}:\\
\indent $\mu_1,\mu _2,\mu _3$ are distinct;\\ 
\indent $\mu_1=\mu_2\neq \mu_3$ and $\epsilon_2-\epsilon_1\geq 0$; \\
\indent $\mu_1=\mu_3\neq \mu_2$ and $\epsilon_3-\epsilon_1\geq 0$;\\
 \indent $\mu_1\neq \mu_2=\mu_3$ and $\epsilon_3-\epsilon_2\geq 0$;\\
 \indent $\mu_1=\mu_2=\mu_3$ and $\epsilon_2-\epsilon_1, \epsilon_3-\epsilon_2\geq 0$. \\
{\it In the sequel we will only consider these cases}.\\

In order to give $\widehat{\mathcal{N}}(\epsilon_*)$ the structure of an algebraic variety we observe that $\mathcal{V}$ has degree -1 and type $O\oplus O\oplus O(-1)$ since
$(\mathcal{V},D)$ is irreducible. We identify $\mathcal{V}$ with $Oe_1\oplus Oe_2\oplus O(-[\infty])e_3$.
The matrix of $D$ with respect to the basis $e_1,e_2,e_3$ has trace zero and is denoted by
\begin{tiny} $\left(\begin{array}{ccc}a_0+a_1z &a_2+a_3z & a_4+a_5z+a_6z^2 \\ a_7+a_8z &a_9+a_{10}z &a_{11}+a_{12}z+a_{13}z^2 \\ a_{14}&a_{15} & a_{16}+a_{17}z \end{array}\right)$.\end{tiny}
The generator $Y=\sum _{n\geq 0}Y_nz^n$ of $L_1$ with $Y_n=Y_n(1)e_1+Y_n(2)e_2+Y_n(3)e_3$
is unique up to multiplication by a constant. The generator $Z=\sum _{n\geq 0}Z_nz^n$ of $L_2$ with 
$Z_n=Z_n(1)e_1^*+Z_n(2)e_2^*+Z_n(3)e_3^*$ is unique up to multiplication by a constant. The  
$Y_*(*),Z_*(*)$ are regarded as homogeneous coordinates.
  
The space $\mathcal{A}$ is defined by the indeterminates $a_*,Y_*(*),Z_*(*)$  and the relations induced by the above requirements. We note that for given $\epsilon_*$, {\it such that the above restrictions
are satisfied}, the 
$Y_n(*),Z_n(*)$ are for $n\geq 1$ eliminated by the relations. Thus $\mathcal{A}$ is an algebraic variety.

 The group $G$ of the automorphisms of $\mathcal{V}$ act upon $\mathcal{A}$. By construction, the set theoretic quotient $\mathcal{A}(\mathbb{C})/G$ coincides with ${\bf S}^+(\epsilon_*)$. Thus  the analytic map
$R:\mathcal{A}\rightarrow \mathcal{R}^+(\mu_*)\times T$, where $\mu_j=e^{2\pi i \epsilon_j}$ for $j=1,2,3$,   
is surjective and $R(\xi_1)=R(\xi_2)$ if and only if there is a $g\in G$ with $g\xi_1=\xi_2$.\\

\noindent  A long MAPLE session verifies: {\it $\mathcal{A}$ has a smooth geometric quotient by $G$}.\\ 
This quotient is by definition $\widehat{\mathcal{N}}(\epsilon_*)$ and the extended Riemann--Hilbert map $\widehat{\mathcal{N}}(\epsilon_*)\rightarrow \mathcal{R}^+(\mu_*)\times T$ is an analytic isomorphism.
 
\begin{small}

\end{small}

\end{document}